\documentclass{article}%
\usepackage{amsfonts}
\usepackage{amsmath}
\usepackage{amssymb}
\usepackage{graphicx}%
\usepackage{color}

\usepackage[nottoc]{tocbibind}
\usepackage{hyperref}
\usepackage{makeidx}
\makeindex
\newcommand\dela[1]{}

\providecommand{\U}[1]{\protect\rule{.1in}{.1in}}
\newtheorem{theorem}{Theorem}[section]

\newtheorem{corollary}[theorem]{Corollary}

\newtheorem{definition}[theorem]{Definition}

\newtheorem{lemma}[theorem]{Lemma}

\newtheorem{proposition}[theorem]{Proposition}
\newtheorem{remark}[theorem]{Remark}

\newenvironment{proof}[1][Proof]{\noindent\textbf{#1.} }{\ \rule{0.5em}{0.5em}}

\numberwithin{equation}{section}

\title{Conservative interacting particles system with anomalous rate of ergodicity.
\thanks{
{Supported by EPSRC %GR/R90994/01 \&
EP/D05379X/1}} %$\phantom{AAAAAAAAAAAAAAAAAAAAAAAAAAAAAAA}$
 }

\author{Z. BRZE\'{Z}NIAK$^\ddag$, F. FLANDOLI$^\S$, M. NEKLYUDOV$^\ddag$,\\
B. ZEGARLI\'NSKI$^\sharp${\footnote{ On leave from Imperial College London}}
\\
$^\S$Dipartimento di Matematica Applicata, Universit\`{a} di Pisa, Italy \\ %} %\\
$^\ddag$Department of Mathematics, University of York, Heslington,
UK
\\
$^\dag$CNRS, Toulouse, France
}
\date{}

\begin{document}
\maketitle

\begin{abstract}
\noindent We analyze certain conservative interacting particle system
and establish
ergodicity of the system for a family of invariant measures. Furthermore, we show that convergence rate to equilibrium is exponential.
This result is of interest because
it presents counterexample to the standard assumption of physicists
that conservative system implies polynomial rate of convergence.

\end{abstract}

\section{Introduction}

In this paper we present an example of the conservative interacting particle system
with exponential rate of convergence to equilibrium.
This system naturally appears in the dyadic model of turbulence (see \cite{B-F-M-2010}).
In \cite{B-F-M-2010} it has been established that the system has anomalous dissipation.
This result seems to be the reason behind exponential rate of convergence to equilibrium.
Similar systems naturally appear in the models of heat conduction  and quantum spin chains (\cite{B2007},\cite{B2008},\cite{F-N-O},\cite{G-K-R},\cite{G-K-R-V}).

Ergodic properties of systems of interacting particles is one of the central topics of statistical mechanics.
They have been studied starting from the works of Spitzer \cite{Spi69} and Dobrushin \cite{Dob71}.
The literature of the subject is huge and we will not attempt to list it here, see \cite{Lig04} and references therein.

Interacting particle systems are usually divided into two classes: conservative and nonconservative ones. Conservative ones are presumed to have at most polynomial rate of convergence to equilibrium and dissipative ones
exponential one (\cite{LY2003}).

In the same time rigorous mathematical results about rates of convergence to equilibrium of conservative systems
has been established only in the handful of cases such as Kawasaki dynamics (\cite{BZ99},\cite{BZ99b}), Ginzburg-Landau type processes (\cite{JLQY1999},\cite{LY2003}) and Brownian moment
processes (\cite{I-N-Z}). The result of this paper shows that existence of
formal conservation law does not necessarily imply polynomial rate of convergence. Consequently, "meta" theorem that conservative interacting particle systems are ergodic with polynomial  rate of convergence to equilibrium is not correct.

\section{The system}

Let $\left(  \Omega,F_{t},P\right)  $ be a filtered probability space and
$\left(  W_{n}\right)  $ be a sequence of independent Brownian motions.
Consider the equation%
\begin{equation}
dX_{n}=k_{n-1}X_{n-1}\circ dW_{n-1}-k_{n}X_{n+1}\circ dW_{n},\qquad
X_{n}\left(  0\right)  =X_{n}^{\left(  0\right)  }\label{system}%
\end{equation}
for all $n\geq 1$, with $X_{0}=0$, $k_{0}=0$, and $k_{n}=\lambda^n, n\in\mathbb{N}$ for some
$\lambda>1$, $X_{n}^{\left(  0\right)  }$ deterministic or $F_{0}$-adapted. The stochastic integral in the system
\eqref{system} is in Stratonovich sense.
\begin{remark}
The assumption $k_n=\lambda^n, n\in\mathbb{N}$ has been imposed for simplicity. It can be relaxed to the assumption
that the sequence $\left\{\frac{k_{n+1}}{k_n}\right\}_{n=1}^{\infty}$ is nondecreasing and the first term of the sequence is bigger than $1$.
\end{remark}
Consider the space
\[
W=\left\{  \left(  x_{n}\right)  _{n\in\mathbb{N}}:\left\Vert x\right\Vert
_{W}^{2}:=\sum k_{n}^{-2}x_{n}^{2}<\infty\right\}  .
\]

\begin{definition}
We say that a sequence of continuous adapted processes $\left(  X_{n}\right)
$ is a weak (in the analytical sense) solution in $W$ of equation
(\ref{system}) if $\left(  X_{n}\right)  $ is $L^{\infty}\left(  \left[
0,T\right]  ;L^{2}\left(  \Omega;W\right)  \right)  $ and
\[
dX_{n}=k_{n-1}X_{n-1}dW_{n-1}-k_{n}X_{n+1}dW_{n}-\frac{1}{2}\left(  k_{n}%
^{2}+k_{n-1}^{2}\right)  X_{n}dt
\]
for each $n\geq1$. If we have%
\[
E\left[  \left\Vert X\left(  t\right)  \right\Vert _{W}^{2}\right]  \leq
E\left[  \left\Vert X^{\left(  0\right)  }\right\Vert _{W}^{2}\right]  ,\qquad
a.e.\text{ }t\geq0
\]
we say it is a Leray solution in $W$.
\end{definition}

\begin{theorem}\label{thm:ExistenceWeakLeray}
For every $X^{\left(  0\right)  }\in L^{2}\left(  \Omega;W\right)  $, $F_{0}%
$-measurable, there exists a weak Leray solution in $W$ of equation
(\ref{system}).
\end{theorem}
\begin{remark}
We use Galerkin type finite dimensional approximation to show existence of solution of system (\ref{system}). Different way would be to apply
results of Holevo \cite{Holevo1996}.
\end{remark}
\begin{proof}
\textbf{Step 1} (existence). For each $N\in\mathbb{N}$, consider the finite
dimensional system%
\[
dX_{n}^{N}=k_{n-1}X_{n-1}^{N}dW_{n-1}-k_{n}X_{n+1}^{N}dW_{n}-\frac{1}%
{2}\left(  k_{n}^{2}+k_{n-1}^{2}\right)  X_{n}^{N}dt,\qquad n=1,...,N
\]
with $X_{0}^{N}=X_{N+1}^{N}=0$ and the initial condition $X_{n}^{N}\left(
0\right)  $ equal to $X_{n}^{\left(  0\right)  }$, $n=1,...,N$. This system
has a unique strong solution, with all moments finite. Indeed, it immediately follows from the Theorem 3.3, p. 7 of
\cite{B-F-M-2009}. Set%
\[
q_{n}^{N}=E\left[  \left(  X_{n}^{N}\right)  ^{2}\right]  .
\]
By It\^{o} formula (we need finite fourth moments to have that the It\^{o}
terms are true martingales, then they disappear taking expected value) we have
(we drop $N$)
\begin{align*}
q_{n}^{\prime} &  =-\left(  k_{n}^{2}+k_{n-1}^{2}\right)  q_{n}+k_{n-1}%
^{2}q_{n-1}+k_{n}^{2}q_{n+1}\\
&  =-k_{n-1}^{2}\left(  q_{n}-q_{n-1}\right)  +k_{n}^{2}\left(  q_{n+1}%
-q_{n}\right)
\end{align*}
for $n=1,...,N$, with $q_{0}=q_{N+1}=0$. Denote by $\left\Vert \cdot
\right\Vert _{W}^{2}$ the same norm introduced above also in the case of a
finite number of components. We have%
\begin{align*}
\frac{d}{dt}E\left[  \left\Vert X^{N}\right\Vert _{W}^{2}\right]   &
=\sum_{n=1}^{N}k_{n}^{-2}\frac{d}{dt}q_{n}^{N}=\\
&  =-\sum_{n=1}^{N}k_{n}^{-2}k_{n-1}^{2}\left(  q_{n}-q_{n-1}\right)
+\sum_{n=1}^{N}k_{n}^{-2}k_{n}^{2}\left(  q_{n+1}-q_{n}\right)  \\
&  =-\lambda^{-2}\sum_{n=2}^{N}\left(  q_{n}-q_{n-1}\right)  +\sum
_{n=1}^{N}\left(  q_{n+1}-q_{n}\right)  \leq -q_{1}.
\end{align*}
Since $q_{1}\geq0$ by definition, we have%
\begin{equation}
E\left[  \left\Vert X^{N}\left(  t\right)  \right\Vert _{W}^{2}\right]  \leq
E\left[  \left\Vert X^{N}\left(  0\right)  \right\Vert _{W}^{2}\right]
,\qquad t\geq0.\label{energy approximated}%
\end{equation}
Thus the sequence $\left(  X^{N}\right)  _{N\geq0}$ is bounded in $L^{\infty
}\left(  \left[  0,T\right]  ;L^{2}\left(  \Omega;W\right)  \right)  $.
Therefore, there exists a subsequence $N_{k}\rightarrow\infty$ such that
$\left(  X_{n}^{\left(  N_{k}\right)  }\right)  _{n\geq1}$ converges weakly to
some $\left(  X_{n}\right)  _{n\geq1}$ in $L^{2}\left(  \Omega\times\left[
0,T\right]  ;W\right)  $ and also weak star in $L^{\infty}\left(  \left[
0,T\right]  ;L^{2}\left(  \Omega;W\right)  \right)  $. Now the proof proceeds
by standard arguments typical of equations with monotone operators (which thus
apply to linear equations), presented in \cite{Pardoux}, \cite{KrylovRoz}. The
subspace of $L^{2}\left(  \Omega\times\left[  0,T\right]  ;W\right)  $ of
progressively measurable processes is strongly closed, hence weakly closed,
hence $\left(  X_{n}\right)  _{n\geq1}$ is progressively measurable. The
one-dimensional stochastic integrals which appear in each equation of the
system are (strongly) continuous linear operators from the subspace of
$L^{2}\left(  \Omega\times\left[  0,T\right]  \right)  $ of progressively
measurable processes to $L^{2}\left(  \Omega\right)  $, hence they are weakly
continuous, a fact that allows us to pass to the limit in each one of the
linear equations of the system. A posteriori, from these integral equations,
it follows that there is a modification such that all components are
continuous. The proof of existence is complete.
%{\color{red}
\begin{remark}\label{rem:l_2_Uniqueness}
If initial condition $X_0\in L^2(\Omega,l^2)$ then there exists unique solution $X\in L^{\infty}\left(  \left[
0,T\right]  ;L^{2}\left(  \Omega;l^2\right)  \right)$ of equation \eqref{system} (see Theorem 3.3 in \cite{B-F-M-2009}).
\end{remark}
%}
\textbf{Step 2} (Leray solution). From (\ref{energy approximated}) and the
definition of $X^{N}\left(  0\right)  $ we have%
\[
E\left[  \left\Vert X^{N}\left(  t\right)  \right\Vert _{W}^{2}\right]  \leq
E\left[  \left\Vert X^{\left(  0\right)  }\right\Vert _{W}^{2}\right]  ,\qquad
t\geq0.
\]
Hence%
\[
E\left[  \int_{a}^{b}\left\Vert X^{N}\left(  t\right)  \right\Vert _{W}%
^{2}dt\right]  \leq\left(  b-a\right)  E\left[  \left\Vert X^{\left(
0\right)  }\right\Vert _{W}^{2}\right]  ,\qquad0\leq a\leq b.
\]
Weak convergence in $L^{2}\left(  \Omega\times\left[  0,T\right]  ;W\right)  $
implies that
\[
E\left[  \int_{a}^{b}\left\Vert X\left(  t\right)  \right\Vert _{W}%
^{2}dt\right]  \leq\left(  b-a\right)  E\left[  \left\Vert X^{\left(
0\right)  }\right\Vert _{W}^{2}\right]  ,\qquad0\leq a\leq b.
\]
By Lebesgue differentiation theorem, we get $E\left[  \left\Vert X\left(
t\right)  \right\Vert _{W}^{2}\right]  \leq E\left[  \left\Vert X^{\left(
0\right)  }\right\Vert _{W}^{2}\right]  $ for a.e. $t$.

\end{proof}
\begin{remark}
The question of uniqueness of weak Leray solutions of system \eqref{system} is an open problem. The \dela{crux of the problem is the} difficulty lies in
 showing  that the difference of two weak non trivial \footnote{non identically zero and non proportional to each other} Leray solutions $X^1$ and $X^2$ is a weak Leray solution itself. It seems that this difficulty \dela{stems from} is related to  the fact that the closed infinite system of linear equations
for a sequence $q_n(t)=\mathbb{E}\left[|X_n^1(t)-X_n^2(t)|^2\right],n=1,\ldots$:
\begin{align}
\label{eqn-q_n}
q_{n}^{\prime}  \dela{ &=-\left(  k_{n}^{2}+k_{n-1}^{2}\right)  q_{n}+k_{n-1}%
^{2}q_{n-1}+k_{n}^{2}q_{n+1}
\\}
&  =-k_{n-1}^{2}\left(  q_{n}-q_{n-1}\right)  +k_{n}^{2}\left(  q_{n+1}%
-q_{n}\right)
%\\q_{n}(0) & = q_n^0, n=1,\ldots
\end{align}
has non-unique solutions in a positive cone of a Banach space space $l^1_k=\{x\in \mathbb{R}^{\mathbb{N}}: \sum\limits_{n=1}^{\infty}\frac{|x_n|}{k_n^2}< \infty\}$ endowed with a norm $|x|_{l_k^1}=\sum\limits_{n=1}^{\infty}\frac{|x_n|}{k_n^2}$.
Indeed, denote $c=2\sup\limits_{t\geq 0}\mathbb{E}\left[||X^1(t)||_W^2+||X^2(t)||_W^2\right]$. Notice that $0<c<\infty$ because $X^1$ and $X^2$ are non identically zero Leray solutions of the system \eqref{system}. Define $p_n=\frac{q_n}{k_n^2c},n\in \mathbb{N}$, $p=(p_n)_{n=1}^{\infty}$.
Then from the identity
\begin{equation}\label{eqn:TotalProb}
\sum\limits_{n=1}^{\infty}p_n(t)=\frac{\mathbb{E}||X^1(t)-X^2(t)||_W^2}{2\sup\limits_{s\geq 0}\mathbb{E}\left[||X^1(s)||_W^2+||X^2(s)||_W^2\right]},
\end{equation}
we can deduce that
\begin{equation}
\sum\limits_{n=1}^{\infty}p_n(t)\leq 1, t\geq 0.\label{eqn:HonCond_1}
\end{equation}
Moreover,
\begin{equation}\label{eqn:HonCond_2}
\exists t_0\geq 0 \mbox{ such that } \sum\limits_{n=1}^{\infty}p_n(t_0)<1.
\end{equation}
Indeed, otherwise it follows from the identity \eqref{eqn:TotalProb} that $X^2=-X^1$ a.s.
Furthermore, because $k_n=\lambda^n$, we have
\begin{align}\label{eqn:TransSystem}
\frac{d}{dt}p&=pA,\\
A&=\left(
\begin{array}{cccccccc}
-k_1^2 & 1 & 0 & 0 & \ldots & \ldots & \ldots & \ldots\\
k_2^2 &  -(k_1^2+k_2^2) & k_1^2 & 0 & \ldots & \ldots & \ldots & \ldots\\
0 & k_3^2 & - (k_2^2+k_3^2) & k_2^2 & \ldots & \ldots & \ldots & \ldots\\
\ldots & \ldots & \ldots & \ldots & \ldots & \ldots & \ldots & \ldots\\
\ldots & \ldots & 0 &k_n^2 & -(k_n^2+k_{n-1}^2) & k_{n-1}^2  & 0 & \ldots \\
\ldots & \ldots & \ldots & \ldots & \ldots & \ldots & \ldots& \ldots
\end{array}
\right).\nonumber
\end{align}
Matrix $A$ has tridiagonal form with positive off diagonal entries. Consequently, we can construct birth and death process $\xi_t,t\geq 0$ on some new probability space $(S,\mathcal{G},\mathbb{P}')$ such that matrix $A$ is a $q$-matrix of the process (see details of construction at p. 8-10 \cite{B-F-M-2010}) and $p_n(t)=\mathbb{P}'(\xi_t=n),n\in \mathbb{N}$.
Now the question of uniqueness of the solution of the system \eqref{eqn:TransSystem} in conjunction with conditions \eqref{eqn:HonCond_1} and \eqref{eqn:HonCond_2} can be reformulated as the question of uniqueness (in law) of the process $\xi_t,t\geq 0$ with given $q$-matrix $A$.
Hence the non uniqueness of the solution of the system \eqref{eqn:TransSystem} with conditions \eqref{eqn:HonCond_1} and \eqref{eqn:HonCond_2}
follow for instance from the criterion in  Theorem 3.2.3, p. 101 of monograph \cite{Anderson} (see also the original paper by Kato \cite{Kato_1954}).
Indeed, we have that the criterion number $S$ is less than infinity for the $q$-matrix $A$ defined above and condition \eqref{eqn:HonCond_2} means that the solution is dishonest\footnote{i.e. $\mathbb{P}'$-a.s. process $\xi$  escapes to infinity in finite time}.
%\coma{But can we prove the lack of uniqueness here?}
\end{remark}
Now we will show existence and uniqueness of solution of system \eqref{system} in a more restrictive class of \textit{moderate} solutions.
We will need following Lemma.
\begin{lemma}
If $X(0)\in L^2(\Omega,l^2)$ then the unique weak solution $X\in L^{\infty}\left(  \left[
0,T\right]  ;L^{2}\left(  \Omega;l^2\right)  \right)$ of the system \eqref{system}
(constructed in Theorem 3.3 in \cite{B-F-M-2009}) is a Leray solution i.e.
\[
E\left[  \left\Vert X\left(  t\right)  \right\Vert _{W}^{2}\right]  \leq
E\left[  \left\Vert X^{\left(  0\right)  }\right\Vert _{W}^{2}\right]  ,\qquad
a.e.\text{ }t\geq0
\]
\end{lemma}
\begin{proof}
We have
\begin{align*}
\frac{d}{dt}\sum\limits_{n=1}^N &\frac{1}{k_n^2}\mathbb{E}\left[|X_n(t)|^2\right]\\
&=-\lambda^{-2}\sum\limits_{n=2}^N(\mathbb{E}\left[|X_n(t)|^2\right]-\mathbb{E}\left[|X_{n-1}(t)|^2\right])+
\sum\limits_{n=1}^N(\mathbb{E}\left[|X_{n+1}(t)|^2\right]-\mathbb{E}\left[|X_n(t)|^2\right])\\
&=\mathbb{E}\left[|X_{N+1}(t)|^2\right]+(\lambda^{-2}-1)\mathbb{E}\left[|X_1(t)|^2\right]-\lambda^2\mathbb{E}\left[|X_N(t)|^2\right]\\
&\leq \mathbb{E}\left[|X_{N+1}(t)|^2\right],N\in \mathbb{N}.
\end{align*}
Consequently, we have
\begin{equation*}
\sum\limits_{n=1}^N\frac{1}{k_n^2}\mathbb{E}\left[|X_n(t)|^2\right]\leq
\sum\limits_{n=1}^N\frac{1}{k_n^2}\mathbb{E}\left[|X_n(0)|^2\right]+\int\limits_0^t\mathbb{E}\left[|X_{N+1}(s)|^2\right]\, ds.
\end{equation*}
Moreover, by $X\in L^{\infty}\left(  \left[
0,T\right]  ;L^{2}\left(  \Omega;l^2\right)  \right)$ we have
\[
\lim\limits_{N\to\infty}\int\limits_0^t\mathbb{E}\left[|X_{N+1}(s)|^2\right]\, ds=0,
\]
and the result follows.
\end{proof}
\begin{corollary}\label{cor:Aux_1}
Assume that the sequence $\{X^{(N)}(0)\}_{N\in\mathbb{N}}\subset L^2(\Omega,l^2)$ of $\mathcal{F}_0$-measurable functions satisfy
\[
(L^2(\Omega,W))-\lim\limits_{N\to\infty}X^{(N)}(0)=X(0)\in L^2(\Omega,W).
\]
Then the sequence $\{X^{(N)}\}_{N\in\mathbb{N}}$ is a Cauchy sequence in $L^{\infty}([0,T],L^2(\Omega,W))$ and, consequently, there exist
$X\in L^{\infty}([0,T],L^2(\Omega,W))$ such that 
\[
X=(L^{\infty}(\left[0,T\right], L^2(\Omega,W)))-\lim\limits_{N\to\infty}X^{(N)}.
\]
\end{corollary}
\begin{corollary}\label{cor:Aux_2}
 Assume that the sequences $\{\widetilde{X}^{(N)}(0)\}_{N\in\mathbb{N}},\{X^{(N)}(0)\}_{N\in\mathbb{N}}\subset L^2(\Omega,l^2)$ of $\mathcal{F}_0$-measurable functions satisfy
\[
(L^2(\Omega,W))-\lim\limits_{N\to\infty}X^{(N)}(0)=X(0)\in L^2(\Omega,W),
\]
\[
(L^2(\Omega,W))-\lim\limits_{N\to\infty}\widetilde{X}^{(N)}(0)=\widetilde{X}(0)\in L^2(\Omega,W),
\]
and
\[
X=(L^{\infty}(\left[0,T\right], L^2(\Omega,W)))-\lim\limits_{N\to\infty}X^{(N)}, \widetilde{X}=(L^{\infty}(\left[0,T\right], L^2(\Omega,W)))-\lim\limits_{N\to\infty}\widetilde{X}^{(N)}.
\]
Then
\[
\mathbb{E}\left[||X(t)-\widetilde{X}(t)||_W^2\right]\leq \mathbb{E}\left[||X(0)-\widetilde{X}(0)||_W^2\right].
\]
\end{corollary}
\begin{proof}
We have by triangle inequality and inequality $(a+b)^2\leq (1+\epsilon)a^2+ (1+\frac{1}{\epsilon})b^2,a,b\in \mathbb{R}\epsilon >0$ that for any $\epsilon>0$ there exists $C(\epsilon)>0$ such that
\begin{align*}
\mathbb{E}\left[||X(t)-\widetilde{X}(t)||_W^2\right]&\leq C(\epsilon) \mathbb{E}\left[||X(t)-X^{(N)}(t)||_W^2\right]\\
&+C(\epsilon)\mathbb{E}\left[||\widetilde{X}(t)-\widetilde{X}^{(N)}(t)||_W^2\right]\\
&+(1+\epsilon)\mathbb{E}\left[||X^{(N)}(t)-\widetilde{X}^{(N)}(t)||_W^2\right].
\end{align*}
Taking limit $N\to \infty$ we conclude that
\begin{align*}
\mathbb{E}\left[||X(t)-\widetilde{X}(t)||_W^2\right]&\leq (1+\epsilon) \lim\limits_{N\to\infty}\mathbb{E}\left[||X^{(N)}(t)-\widetilde{X}^{(N)}(t)||_W^2\right]\\
&\leq (1+\epsilon) \lim\limits_{N\to\infty}\mathbb{E}\left[||X^{(N)}(0)-\widetilde{X}^{(N)}(0)||_W^2\right]\\
&\leq (1+\epsilon) \mathbb{E}\left[||X(0)-\widetilde{X}(0)||_W^2\right],
\end{align*}
and the result follows.
\end{proof}

Now we are ready to define class of moderate solutions.
\begin{definition}
We call $X\in L^{\infty}(\left[0,T\right], L^2(\Omega,W))$ with initial condition $X(0)\in L^2(\Omega,W)$ a moderate solution
of the system \eqref{system} iff $X$ is a weak Leray solution and there exists sequence $X^{(N)}\in L^{\infty}\left(\Omega\times\left[  0,T\right],l_2\right)$ of weak solutions of system \eqref{system} with $X^{(N)}(0)\in L^2(\Omega,l^2)$ such that
\[
X=(L^{\infty}(\left[0,T\right], L^2(\Omega,W)))-\lim\limits_{N\to\infty}X^{(N)}.
\]
\end{definition}
\begin{theorem}\label{thm:ExUniContdep}
Given $X(0)\in L^2(\Omega,W)$ there exists a unique moderate solution of the system \eqref{system}. Furthermore, it depends continuously on its initial condition
in the following sense. If $X^{\eta}$ and $X^{\rho}$ are the solutions
corresponding to the initial conditions $\eta,\rho\in L^{2}\left(
\Omega;W\right)  $, $F_{0}$-measurable, then:

i)
\[
E\left[  \left\Vert X^{\eta}\left(  t\right)  -X^{\rho}\left(  t\right)
\right\Vert _{W}^{2}\right]  \leq E\left[  \left\Vert \eta-\rho\right\Vert
_{W}^{2}\right]  ,\qquad a.e.\text{ }t\geq0
\]

ii) for every $N>0$
\[
E\left[  \sum_{n=1}^{N}k_{n}^{-2}\left(  X_{n}^{\eta}\left(  t\right)
-X_{n}^{\rho}\left(  t\right)  \right)  ^{2}\right]  \leq E\left[  \left\Vert
\eta-\rho\right\Vert _{W}^{2}\right]  ,\qquad\text{for all }t\geq0.
\]
\end{theorem}
\begin{proof}
Existence has been shown in the Corollary \eqref{cor:Aux_1}.
%Theorem \eqref{thm:ExistenceWeakLeray} since the weak Leray solution has been constructed as a limit of Galerkin approximations.
Uniqueness and (i) follow from the Corollary \eqref{cor:Aux_2}. Then (ii)\ holds
by continuity of single components and Fatou theorem.
\end{proof}
\begin{remark}
The same arguments show that the weak Leray solution constructed in the Theorem \eqref{thm:ExistenceWeakLeray} is a moderate solution.
\end{remark}

\section{Markov property}
\label{sec:MarkovProp}

%\begin{proposition}\label{prop:Cont_dependence_1}
%The unique Leray solution in $W$ depends continuously on its initial condition
%in the following sense. If $X^{\eta}$ and $X^{\rho}$ are the solutions
%corresponding to the initial conditions $\eta,\rho\in L^{2}\left(
%\Omega;W\right)  $, $F_{0}$-measurable, then:
%
%i)
%\[
%E\left[  \left\Vert X^{\eta}\left(  t\right)  -X^{\rho}\left(  t\right)
%\right\Vert _{W}^{2}\right]  \leq E\left[  \left\Vert \eta-\rho\right\Vert
%_{W}^{2}\right]  ,\qquad a.e.\text{ }t\geq0
%\]
%
%
%ii) for every $N>0$
%\[
%E\left[  \sum_{n=1}^{N}k_{n}^{-2}\left(  X_{n}^{\eta}\left(  t\right)
%-X_{n}^{\rho}\left(  t\right)  \right)  ^{2}\right]  \leq E\left[  \left\Vert
%\eta-\rho\right\Vert _{W}^{2}\right]  ,\qquad\text{for all }t\geq0.
%\]
%
%\end{proposition}
%
%\begin{proof}
%The difference $X^{\eta}-X^{\rho}$ is a weak solution in $W$ with initial
%condition $\eta-\rho$, hence it is Leray. This implies (i). Then (ii)\ holds
%by continuity of single components and Fatou theorem.
%\end{proof}

We have proved that, for every $x\in W$ there is a unique moderate solution
$\left(  X_{n}^{x}\left(  t\right)  \right)  $ in $W$. Let us prove that the
family $X^{x}$ is a Markov process.
\begin{definition}\label{def:Semigroup}
Define the operator $P_{t}$ on
$B_{b}\left(  W\right)  $ as%
\[
\left(  P_{t}\varphi\right)  \left(  x\right)  :=E\left[  \varphi\left(
X^{x}\left(  t\right)  \right)  \right]  .
\]
By the previous result, $P_{t}$ is well defined also from $C_{b}\left(
W\right)  $ to $C_{b}\left(  W\right)  $.
\end{definition}
\begin{proposition}\label{prop:Markovness}
We have%
\begin{equation}
E\left[  \varphi\left(  X^{x}\left(  t+s\right)  \right)  |F_{t}\right]
=\left(  P_{s}\varphi\right)  \left(  X^{x}\left(  t\right)  \right)
\label{Markov identity}%
\end{equation}
for all $\varphi\in C_{b}\left(  W\right)  $, hence the family $X^{x}$ is a
Markov process. The Markov semigroup $P_{t}$ is Feller.
\end{proposition}

\begin{proof}
We have just to prove the identity (\ref{Markov identity}), the other claims
being obvious or classical. Indeed, Feller property follows from part i) of Theorem \ref{thm:ExUniContdep}.%Proposition \ref{prop:Cont_dependence_1}.
It is enough to show identity (\ref{Markov identity}) when $x\in l^2$. Indeed, general case will follow
from the continuous dependence of moderate solution upon initial condition (part i) of Theorem \ref{thm:ExUniContdep}).

Consider the equation on a generic interval
$\left[  s,t\right]  $ with initial condition $\eta\in L^{2}\left(
\Omega;l^2\right)  $, $F_{s}$-measurable, at time $s$ and call $X^{s,\eta
}\left(  t\right)  $ the solution. Consider the function
\[
Y\left(  t\right)  :=\left\{
\begin{array}
[c]{ccc}%
X^{x}\left(  t\right)   & \text{for} & t\in\left[  0,s\right]  \\
X^{s,X^{x}\left(  s\right)  }\left(  t\right)   & \text{for} & t\geq s.
\end{array}
\right.
\]
Direct substitution into the equations prove that $Y$ is a solution with
initial condition $x$, hence equal to $X^{x}\left(  t\right)  $ also for
$t\geq s$. This proves the evolution property%
\[
X^{s,X^{x}\left(  s\right)  }\left(  t\right)  =X^{x}\left(  t\right)  ,\qquad
t\geq s.
\]
Thus%
\[
E\left[  \varphi\left(  X^{x}\left(  t+s\right)  \right)  |F_{t}\right]
=E\left[  \varphi\left(  X^{t,X^{x}\left(  t\right)  }\left(  t+s\right)
\right)  |F_{t}\right]  .
\]
If we prove that%
\[
E\left[  \varphi\left(  X^{t,\eta}\left(  t+s\right)  \right)  |F_{t}\right]
=\left(  P_{s}\varphi\right)  \left(  \eta\right)
\]
for all $\eta\in L^{2}\left(  \Omega;l^2\right)  $, $F_{t}$-measurable, we are
done. If $\eta=x$, a.s. constant, it is true, by exploiting the fact that the
increments of the Brownian motions $W_{n}$ from $t$ to $t+s$ are independent
of $F_{t}$; and because the dynamics is autonomous. From constant values one
generalizes to $\eta=\sum_{i=1}^{n}x_{i}1_{A_{i}}$, $A_{i}\in F_{t}$;\ indeed,
for such $\eta$, we have%
\[
X^{t,\eta}\left(  t+s\right)  =\sum_{i=1}^{n}X^{t,x_{i}}\left(  t+s\right)
1_{A_{i}}.
\]
Finally we have the identity for all $\eta$ by the continuity result above.
\end{proof}

\section{Invariant measures}

Consider the measures $\mu_{r}$, parametrized by $r\geq0$, formally defined
as
\[
\mu_{r}\left(  dx\right)  =\frac{1}{Z}\exp\left(  -\frac{\sum_{n=1}^{\infty
}x_{n}^{2}}{2r}\right)  dx.
\]
The rigorous definition is: $\mu_{r}$ is the Gauss measure on $l^{2}$, namely
the Gaussian measure on $W$ having covariance equal to identity. For every
function $f$ of the first $n$ coordinates only of $l^{2}$, the measure
$\mu_{r}$ is given by%
\[
\int_{Y}f\left(  x_{1},...,x_{n}\right)  \mu_{r}\left(  dx\right)  =\frac
{1}{Z_{n}}\int_{\mathbb{R}^{n}}f\left(  x_{1},...,x_{n}\right)  \exp\left(
-\frac{\sum_{k=1}^{n}x_{k}^{2}}{2r}\right)  dx_{1}...dx_{n}%
\]
where $Z_{n}=\left(  2\pi r\right)  ^{n/2}$. This formula identifies $\mu_{r}$.

Moreover, for technical reasons, we need
\[
\widetilde{W}=\left\{  \left(  x_{n}\right)  _{n\in\mathbb{N}}:\left\Vert x\right\Vert
_{W}^{2}:=\sum k_{n}^{-4}x_{n}^{2}<\infty\right\}  .
\]
Notice that
\[
\mu_{r}\left(  \widetilde{W}\right)  =1
\]
and the embedding of $W$ in
$\widetilde{W}$ is compact.

\begin{proposition}
For every $r>0$, $\mu_{r}$ is invariant for the Markov semigroup $P_{t}$
defined above on $W$.
\end{proposition}

\begin{proof}
It is sufficient to prove
\[
\int_{Y}\left(  P_{t}\varphi\right)  \left(  x\right)  \mu_{r}\left(
dx\right)  =\int_{Y}\varphi\left(  x\right)  \mu_{r}\left(  dx\right)
\]
for all $\varphi$ of the form $\varphi\left(  x\right)  =f\left(
x_{1},...,x_{n}\right)  $, with bounded continuous $f$. We have%
\[
\int_{Y}\left(  P_{t}\varphi\right)  \left(  x\right)  \mu_{r}\left(
dx\right)  =E\left[  \int_{Y}f\left(  X_{1}^{x}\left(  t\right)
,...,X_{n}^{x}\left(  t\right)  \right)  \mu_{r}\left(  dx\right)  \right]  .
\]
The strategy now is the following one. On an enlarged probability space, if
necessary, we can define an $F_{0}$-measurable r.v. $X^{\left(  0\right)  }\in
L^{2}\left(  \Omega;W\right)  $ with law $\mu_{r}$. For every $N>0$ denote by
$X^{N}$ the Galerkin approximations used to prove existence above, with
initial condition $X_{n}^{N}\left(  0\right)  =X_{n}^{\left(  0\right)  }$.
Sequence $X^{N}$ weakly
converges to the moderate solution $X$ having initial condition $X^{\left(
0\right)  }$. Denote by $\rho_{N}$ the law of $X^{N}$ and by $\rho$ the law of
$X$, on $L^{2}\left(  \left[  0,T\right];W\right)  $. We shall
prove that the sequence $\rho_{N}$ is tight in $L^{2}\left(  \left[
0,T\right];\widetilde{W}\right)  $. Then there exists a
subsequence $\rho_{n_{k}}$ weakly convergent to some probability measure on
$L^{2}\left(  \left[  0,T\right];\widetilde{W}\right)  $. Such
measure must be $\rho$.

It is enough to show that the sequence $X^{N}$ of Galerkin approximations is bounded in
$L^2(\Omega,L^2(\left[0,T\right],W))\cap L^2(\Omega,W^{\alpha,2}(\left[0,T\right],\widetilde{W})),\alpha\in (0,1)$. That implies that
laws $\{\rho_N\}_{N=1}^{\infty}$ are bounded in probability
on
\[
L^2(\left[0,T\right],W)\cap W^{\alpha,2}(\left[0,T\right],\widetilde{W}),\alpha\in (0,1).
\]

Since embedding
\[
L^2(\left[0,T\right],W)\cap W^{\alpha,2}(\left[0,T\right],\widetilde{W})\subset L^2(\left[0,T\right],\widetilde{W}),\alpha\in (0,1).
\]
is compact by Theorem 2.1, p. 370 of \cite{FlandoliGatarek_1995}( applied with $B_0=W$, $B=B_1=\widetilde{W}$, $p=2$) we shall conclude that the sequence $\rho_{N}$ is tight in $L^{2}\left(  \left[
0,T\right];\widetilde{W}\right)  $.

Since the sequence $(X^{N})$ is bounded in $L^{\infty}(\left[0,T\right],L^2(\Omega,W))$ it remains to show that the sequence $(X^{N})$ is bounded in $L^2(\Omega,W^{\alpha,2}(\left[0,T\right],\widetilde{W}))$ for some $\alpha\in (0,1)$.

Decompose $X^N$ as
\[
X^N(t)=X^N(0)-\int_0^tA^NX^N(s)ds+\int_0^tB^N(X^N)dW(s)=J_1^N(t)+J_2^N(t)+J_3^N(t)
\]
where
\begin{align*}
(A^Nx)_{n,m}&=-\frac{\delta_{n,m}}{2}(k_{n-1}^2+k_n^2)x_n,\\
(B^Nx)_{n,m}&=k_{n-1}x_{n-1}\delta_{n,m+1}-k_{n}x_{n+1}\delta_{n,m},x\in P_N(W),\,m,n=1,\ldots,N.
\end{align*}
We have
\begin{equation}\label{eqn:Aux_1}
\mathbb{E}|J_N^1|_{W^{1,2}(0,T;W)}^2\leq T\mathbb{E}|X^{(0)}|_{W}^2.
\end{equation}
Since $|A^N|_{\mathcal{L}(W,\widetilde{W})}\leq K=1+\sup\limits_{n}\frac{k_{n-1}^2}{k_n^2}$ we infer that
\begin{equation}\label{eqn:Aux_2}
\mathbb{E}|J_2^N|_{W^{1,2}(0,T;\widetilde{W})}^2\leq C(T,K)\mathbb{E}|X^N|_{L^2(\left[0,T\right],W)}^2\leq C(T,K)\mathbb{E}|X^{(0)}|_{W}^2.
\end{equation}
Fix $\alpha\in (0,\frac{1}{2})$.
By Lemma 2.1, p. 369 of \cite{FlandoliGatarek_1995} we have that
\begin{equation}\label{eqn:Aux_3}
\mathbb{E}|J_2^N|_{W^{\alpha,2}(0,T;\widetilde{W})}^2\leq\mathbb{E}\int_0^T|B^N(X^N)|_{L_{HS}(l^2,\widetilde{W})}^2\,ds
\end{equation}
Notice that
\begin{align}
|B(x)|_{L_{HS}(l^2,\widetilde{W})}^2&=\sum\limits_{n=1}^{\infty}|B(x)(e_n)|_{\widetilde{W}}^2\nonumber\\
&\leq \sum\limits_{n=1}^{\infty} k_n^{-4}k_n^2x_{n+1}^2+
k_{n+1}^{-4}k_n^2x_n^2\leq (K+K^2)|x|_{W}^2,x\in W.\label{eqn:Aux_4}
\end{align}
where $(e_n)_{n=1}^{\infty}$ is ONB in $l^2$.

Combining inequalities \eqref{eqn:Aux_3} and \eqref{eqn:Aux_4} we infer that
\begin{equation}\label{eqn:Aux_5}
\mathbb{E}|J_2^N|_{W^{\alpha,2}(0,T;\widetilde{W})}^2\leq C\mathbb{E}\int_0^T|X^N(s)|_{W}^2\,ds\leq C(T,K,\alpha)\mathbb{E}|X^{(0)}|_{W}^2.
\end{equation}

Hence, inequalities \eqref{eqn:Aux_1}, \eqref{eqn:Aux_2} and \eqref{eqn:Aux_5} imply that for some $\alpha\in(0,\frac{1}{2})$
\begin{equation}\label{eqn:Aux_6}
\mathbb{E}|X^N|_{W^{\alpha,2}(\left[0,T\right],\widetilde{W})}^2\leq C(T,K,\alpha)\mathbb{E}|X^{(0)}|_{W}^2,
\end{equation}
and the result follows.
\end{proof}
\begin{corollary}\label{cor:SemigroupReversibility}
The semigroup $(P_t)_{t\geq0}$ acting on $C_b(W)$ can be extended to $L^p(W,\mu_r)$ for any $p\geq1$.
Generator of the semigroup $(P_t)_{t\geq 0}$ is given by the formula
\[
\mathcal{L}=\frac{1}{2}\sum\limits_{j=1}^{\infty}k_{j}^2D_{j,j+1}^2
\]
with $D_{j,j+1}=x_j\partial_{x_{j+1}}-x_{j+1}\partial_{x_j}, j\in\mathbb{N}$.
\end{corollary}
\begin{proof}
It follows from It\^o formula.
\end{proof}

\section{Symmetry of the generator in the Sobolev spaces}
Let
\[
L=\sum\limits_{i=1}^{\infty}(\frac{\partial^2}{\partial x_i^2}-x_i\frac{\partial}{\partial x_i})
\]
be Ornstein-Uhlenbeck operator and
\[
\mathcal{H}^n=\left\{f\in L^2(W,\mu_r):|f|_{\mathcal{H}^n}^2=|f|_{L^2(W,\mu_r)}^2+((-L)^nf,f)_{L^2(W,\mu_r)}<\infty
\right\},n\in\mathbb{Z},
\]
\[
\mathcal{C}=\left\{\phi: W\to\mathbb{R}, \phi(x)=f(x_1,\ldots,x_n), f\in C^4(\mathbb{R}^n,\mathbb{R}),n\in\mathbb{N}\right\}.
\]
We have
\[
[D_{i,i+1},L]\phi=0,\phi\in\mathcal{C}.
\]
Consequently,
\[
[\mathcal{L},L]\phi=0,\phi\in\mathcal{C},
\]
and
\begin{proposition}
For all $f,g\in\mathcal{C}$ we have
\[
(f,\mathcal{L}g)_{\mathcal{H}^n}=(g,\mathcal{L}f)_{\mathcal{H}^n}=-\sum\limits_{l=1}^{\infty}k_l^2(D_{l,l+1}f,D_{l,l+1}g)_{\mathcal{H}^n},n\in\mathbb{Z}.
\]
\end{proposition}
Fix $n\in\mathbb{N}\cup{0}$.
\begin{corollary}\label{cor:Contractivity}
The operator $\mathcal{L}$ is closable in $\mathcal{H}^n$ and its closure has bounded from above self-adjoint extension, which we continue to denote by the same symbol
$\mathcal{L}$. Moreover, the
self-adjoint extension $\mathcal{L}$ generates a strongly continuous contraction semigroup $T_t=e^{t\mathcal{L}}:\mathcal{H}^n\to \mathcal{H}^n$ such that $T_t=P_t|_{\mathcal{H}^n}$.
\end{corollary}

\section{Ergodicity}

%\coma{Maybe we can recall what is $k_n$ etc?}
Define
\begin{eqnarray}\label{eqn-A_r}
\mathcal{A}_r(f)&=&\sum\limits_{n=1}^{\infty}|\partial_n f|_{L^2(W,\mu_r)}^2=(-Lf,f)_{L^2(W,\mu_r)},
\\
\label{eqn-nu}
\nu&=&\sum\limits_{n=1}^{\infty}\frac{n}{k_n^2},\,(\mbox{where } k_n=\lambda^n, \lambda>1).
\end{eqnarray}
Let us remind the reader that $P_t:L^2(W,\mu_r)\to L^2(W,\mu_r),t\geq 0$ is the semigroup with the infinitesimal generator
\[
\mathcal{L}=\frac{1}{2}\sum\limits_{j=1}^{\infty}k_{j}^2D_{j,j+1}^2.
\]
Its existence, construction and properties has been discussed above, see the definition \ref{def:Semigroup}, the proposition \ref{prop:Markovness} and the corollaries \ref{cor:SemigroupReversibility} and \ref{cor:Contractivity}.
\begin{theorem}\label{thm:ExponentialConv}
There exist $C=C(\lambda)>0$ such that for any $f\in \mathcal{H}^1$ and $t\geq 0$
\begin{equation}\label{ExponentialConv_1}
\mathcal{A}_r(P_tf)\leq C\mathcal{A}_r(f)e^{-\frac{t}{\nu}},f\in\mathcal{H}^1.
\end{equation}
\end{theorem}
\begin{proof}

It is enough to show \eqref{ExponentialConv_1} for $f\in C_b^4(W)$. Indeed, $C_b^4(W)$ is dense in $\mathcal{H}^1$ and
$(P_t)_{t\geq
0}$ is a contraction on $\mathcal{H}^1$ by \ref{cor:Contractivity}.

Denote $f_t=P_tf$ for $t\geq 0$.
For $i\in\mathbb{N}$, we can calculate that
\begin{align}
|\partial_i f_t|^2-P_t|\partial_i f|^2&=\int_0^t\frac{d}{ds}P_{t-s}|\partial_i f_s|^2ds\label{eqn:AuxEquation-1}\nonumber\\
&=\int_0^tP_{t-s}(-\mathcal{L}(|\partial_i f_s|^2)+2\partial_i f_s\mathcal{L}\partial_i f_s+2\partial_i
f_s[\partial_i,\mathcal{L}]f_s)ds\nonumber\\
&=\int_0^tP_{t-s}\Big(-\sum\limits_{m\in\mathbb{N}}k_m^2|D_{m,m+1}(\partial_i f_s)|^2\nonumber\\
& \quad +\partial_i f_s(-(k_{i}^2+k_{i-1}^2)\partial_i f_s+2k_{i-1}^2D_{i,i-1}\partial_{i-1}f_s+2k_i^2D_{i,i+1}\partial_{i+1}f_s)\Big)ds .
\end{align}
Integrating \eqref{eqn:AuxEquation-1} with respect to the invariant measure
$\mu_r$ yields
\begin{align}
\mu_r|\partial_i f_t|^2&-\mu_r|\partial_i f|^2 = \int_0^t\Big(-\sum\limits_{m\in\mathbb{N}}k_m^2|D_{m,m+1}(\partial_i f_s)|^2\nonumber\\
&- (k_{i}^2+k_{i-1}^2)\mu_r|\partial_i f_s|^2+2k_{i-1}^2\mu_r(\partial_i f_sD_{i,i-1}\partial_{i-1}f_s)\nonumber\\
&+ 2k_{i}^2\mu_r(\partial_i f_sD_{i,i+1}\partial_{i+1}f_s)\Big)ds.
\end{align}
Notice that the operators $D_{i,j},i,j\in \mathbb{N}$, are antisymmetric in
$L^2(\mu_r)$. Therefore
\begin{align}
\mu_r|\partial_i f_t|^2&-\mu_r|\partial_i f|^2 = \int_0^t\Big(-\sum\limits_{ m\in\mathbb{N}}k_m^2\mu_r|D_{m,m+1}(\partial_i f_s)|^2\nonumber\\
&- (k_{i}^2+k_{i-1}^2)\mu_r|\partial_i f_s|^2-2k_{i-1}^2\mu_r(D_{i,i-1}(\partial_i f_s)\partial_{i-1}f_s)\nonumber\\
&- 2k_{i}^2\mu_r(D_{i,i+1}(\partial_i f_s)\partial_{i+1}f_s)\Big)ds.
\end{align}
Hence, by Young's inequality we deduce that
\begin{align}
\mu_r|\partial_i f_t|^2&-\mu_r|\partial_i f|^2 \leq \int_0^t\Big(-\sum\limits_{m\in\mathbb{N}}k_m^2\mu_r|D_{m,m+1}(\partial_i f_s)|^2\nonumber\\
&- (k_{i}^2+k_{i-1}^2)\mu_r|\partial_i f_s|^2+k_{i-1}^2\mu_r|D_{i,i-1}\partial_i f_s|^2+k_{i-1}^2\mu_r|\partial_{i-1}f_s|^2\nonumber\\
&+k_{i}^2\mu_r|D_{i,i+1}\partial_i f_s|^2+k_{i}^2\mu_r|\partial_{i+1}f_s)|^2\Big)ds\leq\nonumber\\
&\leq\int_0^t\Big(-\sum\limits_{m\neq i,i-1}k_m^2\mu_r|D_{m,m+1}(\partial_i f_s)|^2\nonumber\\
&- (k_{i}^2+k_{i-1}^2)\mu_r|\partial_i f_s|^2+k_{i-1}^2\mu_r|\partial_{i-1}f_s|^2+k_i^2\mu_r|\partial_{i+1}f_s)|^2\Big)ds.\label{ComparBound_1}
\end{align}
Let $\triangle^{k}$  denote the operator on $\mathbb{R}^{\mathbb{N}}$ given by
\begin{equation}\label{eqn:LaplOperatorDef}
\triangle^{k}f(i)=k_i^2(f(i+1)-f(i))+k_{i-1}^2(f(i-1)-f(i)),i\in\mathbb{N},f:\mathbb{N}\to\mathbb{R},k_0=0,
\end{equation}
and set $F(i,t)=\mu_r|\partial_i(P_tf)|^2$ for $t\geq 0, i\in \mathbb{N}$.
Then we can rewrite \eqref{ComparBound_1} as
\begin{equation}
F(t)\leq
F(0)+\int\limits_0^t\triangle^{k} F(s)\,ds,\qquad t\in[0,\infty).\label{ComparBound_2}
\end{equation}
Hence, by the positivity of the semigroup $(e^{t\triangle^{k}})_{t\geq 0}$, and Duhamel's principle, we can conclude that
\begin{equation}
F(t)\leq G(t):=e^{t\triangle^{k}} F(0)\label{ComparBound_3}
\end{equation}
for $t\in[0, \infty)$.
It has been shown in \cite{B-F-M-2010} that there exist $C=C(\{k_n\}_{n=1}^{\infty})>0$ such that
\begin{equation}\label{ComparBound_4}
\sum\limits_{i}G(i,t)\leq Ce^{-\frac{t}{\nu}}\sum\limits_{i}G(i,0),t\geq 0\dela{,\nu=\sum\limits_{n=1}^{\infty}\frac{n}{k_nq
^2}}.
\end{equation}
where $\nu$ has been defined in\eqref{eqn-nu}.
Now the result follows from inequalities \eqref{ComparBound_3} and \eqref{ComparBound_4}.
\end{proof}
\begin{corollary}
\[
\mu_r(P_tf-\mu_rf)^2\leq C\mathcal{A}_r(f)e^{-\frac{t}{\nu}},f\in\mathcal{H}^1.
\]
\end{corollary}
\begin{proof}
The proof of this result  follows  immediately from the Poincar\'e inequality for the Gaussian measure $\mu_r$, see, for instance,  theorem 5.5.1 , p. 226 in \cite{Bogachev}. %\coma{We need a reference for the last part of the sentence.}
\end{proof}

Define
\[
\overline{\mathcal{H}}^1=\{f\in L^2(\mu_r)|\int fd\mu_r=0,||f||_{\overline{\mathcal{H}}^1}^2=\mathcal{A}_r(f)<\infty\}.
\]
Let $D_{\overline{\mathcal{H}}^1}(\mathcal{L})$ domain of operator $\mathcal{L}$ in $\overline{\mathcal{H}}^1$.
\begin{corollary}[Poincare inequality in $\overline{\mathcal{H}}^1$]\label{cor:PoincareIneq_1}
There exists $C>0$ such that
\[
||f||_{\overline{\mathcal{H}}^1}^2\leq C <-\mathcal{L}f,f>_{\overline{\mathcal{H}}^1}, f\in D_{\overline{\mathcal{H}}^1}(\mathcal{L}).
\]
\end{corollary}
\begin{corollary}\label{cor:PoincareIneq_2}
There exists $C>0$ such that
\[
\mu_r(f-\mu_r f)^2\leq 2\nu (-\mathcal{L}f,f)_{L^2(\mu_r)}(1+max(0,\log{\frac{C ||f||_{\overline{\mathcal{H}}^1}^2}{2\nu(-\mathcal{L}f,f)_{L^2(\mu_r)}}})), f\in D(\mathcal{L})\cap\mathcal{H}^1.
\]
\end{corollary}
Corollaries \ref{cor:PoincareIneq_1} and \ref{cor:PoincareIneq_2} can be deduced from the theorem \ref{thm:ExponentialConv}
in the same way as Nash-Liggett inequalities has been proven in \cite{I-N-Z}, see proof of Theorem 8.1.
\begin{remark}
The convergence in Theorem \ref{thm:ExponentialConv} cannot be improved. Indeed, let $S(l,t)=P_t(x_l^2)$ for $t\geq 0$ and $l\in\mathbb{N}$.
Then $\mathcal{L}x_l^2=k_l^2(x_{l+1}^2-x_l^2)+k_{l-1}^2(x_{l-1}^2-x_l^2),l\in\mathbb{N}$, so that,
$$
\frac{\partial S}{\partial t}=\triangle^k S,
$$
where $\triangle^k$ is defined by formula \eqref{eqn:LaplOperatorDef}.
Thus
\begin{equation}
S(t)=e^{t\triangle^k}S(0),t\geq 0,\label{eqn:ExponentialGrowth}
\end{equation}
so that convergence rate in the
Theorem \ref{thm:ExponentialConv} is achieved.
\end{remark}
\begin{remark}
If we consider the system \eqref{system} with the choice of positive $\{k_n\}_{n=1}^{\infty}$ such that $\sum\limits_{n=1}^{\infty}\frac{n}{k_n^2}=\infty$ then it is possible to show polynomial rate of convergence to equilibrium for the corresponding semigroup $P_t,t\geq 0$ %\coma{but of what? Maybe we can formulate a result?}
in the same way as in the paper \cite{I-N-Z}.
\end{remark}
\begin{remark}
We have shown that the exponential rate of convergence for the semigroup $(P_t)_{t\geq 0}$ holds
if  $k_n=\lambda^n$, $\lambda>1$. In the same time, there is no spectral gap if $\lambda=1$.
Indeed, it is enough to notice that if $f_N=\sum\limits_{k=1}^N(x_k^2-1)$ then
\[
||f_N||_{\overline{\mathcal{H}}^1}^2\sim N,<-\mathcal{L}f_N,f_N>_{\overline{\mathcal{H}}^1}\mbox{is independent upon $N$},
\]
and corollary $13$ does not hold. Thus, the asymptotic behaviour of our conservative system depends upon the value of the parameter $\lambda$.
\end{remark}

\bibliography{Existence250710}

\end{document}